\documentclass{article}

\usepackage{amsfonts}
\usepackage{amsmath}
\usepackage{amssymb}
\usepackage{amsthm}
\usepackage{booktabs}
\usepackage[justification=centering]{caption}
\usepackage[margin=1.25in]{geometry}
\usepackage{graphicx}
\usepackage{hyperref}
\usepackage[noabbrev,nameinlink]{cleveref}
\usepackage{lineno}
\usepackage[numbers,sort&compress]{natbib}
\usepackage{pifont}
\usepackage{tikz}
\usepackage{ulem}

\crefformat{equation}{(#2#1#3)}
\hypersetup{colorlinks,breaklinks}
\hypersetup{
	colorlinks=true,
	linkcolor=[rgb]{0,0,0.8},
	filecolor=blue,      
	urlcolor=magenta,
	citecolor=[rgb]{0,0.6,0},
}

\newcommand{\st}{\mathrm{s.}~\mathrm{t.}}
\newcommand{\lrbrace}[1]{\left\{#1\right\}}
\newcommand{\lrbracket}[1]{\left(#1\right)}
\newcommand{\T}{\top}
\newcommand{\trace}{\mathrm{Tr}}
\newcommand{\eps}{\varepsilon}
\newcommand{\calF}{\mathcal{F}}
\newcommand{\mb}[1]{\mathbf{#1}}

\newcommand{\snorm}[1]{\Vert#1\Vert}
\newcommand{\abs}[1]{\left|#1\right|}
\newcommand{\inner}[1]{\left\langle#1\right\rangle}
\newcommand{\R}{\mathbb{R}}
\newcommand{\N}{\mathbb{N}}
\newcommand{\dist}{\mathrm{dist}}
\newcommand{\rbd}{\mathrm{rbd}~}
\newcommand{\h}{\mathbf{h}}
\newcommand{\w}{\mathbf{w}}
\newcommand{\x}{\mathbf{x}}
\newcommand{\y}{\mathbf{y}}
\newcommand{\z}{\mathbf{z}}
\newcommand{\vv}{\mathrm{vec}}
\newcommand{\supp}{\mathrm{supp}}
\newcommand{\one}{\mathbf{1}}
\newcommand{\opt}{\mathrm{opt}}
\makeatletter
\newcommand{\labeltarget}[1]{\Hy@raisedlink{\hypertarget{#1}{}}}
\makeatother

\DeclareMathOperator*{\argmin}{arg\,min}

\newtheorem{assume}{Assumption}

\newtheorem{lem}{Lemma}
\newtheorem{thm}{Theorem}

\Crefname{assume}{Assumption}{Assumptions}
\Crefname{coro}{Corollary}{Corollarys}
\Crefname{figure}{Figure}{Figures}
\Crefname{lem}{Lemma}{Lemmas}
\Crefname{thm}{Theorem}{Theorems}
\crefname{subsection}{subsection}{subsections}

\numberwithin{equation}{section}

\title{The Exactness of the $\ell_1$ Penalty Function for a Class of Mathematical Programs with Generalized Complementarity Constraints}
\date{\today}
\author{Yukuan Hu\thanks{State Key Laboratory of Scientific and Engineering Computing, Academy of Mathematics and Systems Science, Chinese Academy of Sciences, Beijing 100190, China and University of Chinese Academy of Sciences, Beijing 100049, China (\href{mailto:ykhu@lsec.cc.ac.cn}{ykhu@lsec.cc.ac.cn}, \href{mailto:liuxin@lsec.cc.ac.cn}{liuxin@lsec.cc.ac.cn}).} \and Xin Liu\footnotemark[1]}

\begin{document}
	
	\maketitle
	
	\begin{abstract}
	In a Mathematical Program with Generalized Complementarity Constraints (MPGCC), complementarity relationships are imposed between each pair of variable blocks. MPGCC includes the traditional Mathematical Program with Complementarity Constraints (MPCC) as a special case. On account of the disjunctive feasible region, MPCC and MPGCC are generally difficult to handle. The $\ell_1$ penalty method, often adopted in computation, opens a way of circumventing the difficulty. Yet it remains unclear about the exactness of the $\ell_1$ penalty function, namely, whether there exists a sufficiently large penalty parameter so that the penalty problem shares the optimal solution set with the original one. In this paper, we consider a class of MPGCCs that are of multi-affine objective functions. This problem class finds applications in various fields, e.g., the multi-marginal optimal transport problems in many-body quantum physics and the pricing problem in network transportation. We first provide an instance from this class, the exactness of whose $\ell_1$ penalty function cannot be derived by existing tools.	We then establish the exactness results under rather mild conditions. Our results cover those existing ones for MPCC and apply to multi-block contexts.
	\end{abstract}

	\section{Introduction}
	
	\par A Mathematical Program with Complementarity Constraints (MPCC) 
	takes the form 
	\begin{equation}
		\begin{array}{cl}
			\min\limits_{\x,\y} & g(\x,\y)\\
			\st & \h(\x,\y)\le0,\\
			 & \x\ge0,~\y\ge0,~\inner{\x,\y}=0,
		\end{array}
		\label{eqn:general MPCC}
	\end{equation}
	where $\x$, $\y\in\R^m$, $g:\R^{2m}\to\R$, and $\h:\R^{2m}\to\R^q$. MPCC has found wide applications in economics and engineering design. For a review on this topic, interested readers may refer to \cite{luo1996mathematical,ralph2007nonlinear} and the references therein. If $g$ and $\h$ are affine, the MPCC \cref{eqn:general MPCC} reduces to a Linear Program with Complementarity Constraints (LPCC). For problems with multiple variable blocks, 
	the complementarity relationships imposed between each block pair leads to the following Mathematical Program with Generalized Complementarity Constraints (MPGCC) 
	\begin{equation}
		\begin{array}{cl}
			\min\limits_{\x_1,\ldots,\x_n} & \tilde g(\x_1,\ldots,\x_n)\\
			\st & \tilde \h(\x_1,\ldots,\x_n)\le0,\\
			 & \x_i\ge0,~i=1,\ldots,n,\\
			 & \inner{\x_i,\x_j}=0,~\forall~i,j\in\{1,\ldots,n\}:i\ne j,
		\end{array}
		\label{eqn:general MPGCC}
	\end{equation}
	where $\x_i\in\R^m$ for $i=1,\ldots,n$, $\tilde g:\R^{mn}\to\R$, and $\tilde\h:\R^{mn}\to\R^q$.
	
	\par It is worth mentioning that MPCC \cref{eqn:general MPCC} and MPGCC \cref{eqn:general MPGCC} violate standard constraint qualifications at any feasible point, such as the Mangasarian-Fromovitz constraint qualification (MFCQ) \cite{flegel2005guignard}. Consequently, the local solutions, and of course optimal solutions, do not necessarily satisfy the associated Karush-Kuhn-Tucker conditions, which renders these problems particularly difficult to cope with. It can also be proved that globally solving a general MPCC or MPGCC is NP-complete \cite{chung1989np,hansen1992new,jeroslow1985polynomial}. To this end, several tailored constraint qualifications and stationarity notions have been proposed for MPCC; see \cite{flegel2005abadie,flegel2005guignard,flegel2006direct,guo2021mathematical,outrata1999optimality,scheel2000mathematical,ye2005necessary} for example. 
	
	\par Instead of facing the original MPCC \cref{eqn:general MPCC} or MPGCC \cref{eqn:general MPGCC}, one could alternatively penalize the (generalized) complementarity constraints using $\ell_1$ penalty term, and then turn to consider the penalty problem. Taking the MPGCC \cref{eqn:general MPGCC} as an instance, we write its $\ell_1$ penalty counterpart as 
	\begin{equation}
		\begin{array}{cl}
			\min\limits_{\x_1,\ldots,\x_n} & \tilde g_\beta(\x_1,\ldots,\x_n):=\tilde g(\x_1,\ldots,\x_n)+\beta p(\x_1,\ldots,\x_n)\\
			\st & \tilde\h(\x_1,\ldots,\x_n)\le0;~\x_i\ge0,~i=1,\ldots,n,
		\end{array}
		\label{eqn:general MPGCC penalty}
	\end{equation}
	where $\beta>0$ refers to the penalty parameter and the $\ell_1$ penalty term
	\begin{equation}
		p(\x_1,\ldots,\x_n):=\sum_{i=1}^n\sum_{j>i}\inner{\x_i,\x_j},\quad\forall~\x_1,\ldots,\x_n\in\R^m.
		\label{eqn:penalty term}
	\end{equation}
	The penalty problem \cref{eqn:general MPGCC penalty} enjoys several nice properties: (i) there is no need to take absolute value for $p$ because of the nonnegative constraints and, therefore, no nonsmoothness is introduced; (ii) the problem \cref{eqn:general MPGCC penalty} is free of complementarity constraints and standard constraint qualifications may readily hold. Nevertheless, it remains unclear about the exactness of the $\ell_1$ penalty function, i.e., whether or not the \textit{optimal solution sets} of \cref{eqn:general MPGCC} and \cref{eqn:general MPGCC penalty} coincide for all sufficiently large $\beta$.
	
	\par In this paper, we consider a class of MPGCC that reads
	\begin{equation}
		\begin{array}{cl}
			\displaystyle\min_{\x_1,\ldots,\x_n} & \displaystyle f(\x_1,\ldots,\x_n)\\
			\st & \x_i\in\Omega_i,~\x_i\ge0,~i=1,\ldots,n,\\
			& \inner{\x_i,\x_j}=0,~\forall~i,j\in\{1,\ldots,n\}:i\ne j.
		\end{array}
		\tag{P}
		\labeltarget{MPGCC}
		\label{eqn:MPGCC}
	\end{equation}
	Here $f:\R^{mn}\to\R$ is \textit{multi-affine}, i.e., for $i=1,\ldots,n$, it is affine with respect to $\x_i$ after fixing the other $n-1$ blocks. The sets $\{\Omega_i\}_{i=1}^n$ are polyhedrons in $\R^{m}$. In the sequel, we denote the feasible region of \cref{eqn:MPGCC} by $\calF$ and let $\z:=(\x_1,\ldots,\x_n)\in\R^{mn}$, $\Omega_i^+:=\Omega_i\cap\R^m_+$ ($i=1,\ldots,n$) for brevity.
	
	\par The relationships among the general LPCC, MPCC, MPGCC as well as the scope of the model \cref{eqn:MPGCC} are depicted in \Cref{fig:relationship}. The cyan ellipsoid with solid boundary stands for MPGCC, the larger cyan disk with dashed boundary for MPCC, and the smaller cyan disk with dotted boundary for LPCC. The red ellipsoid with dashdotted boundary refers to the scope of the model \cref{eqn:MPGCC}. 
	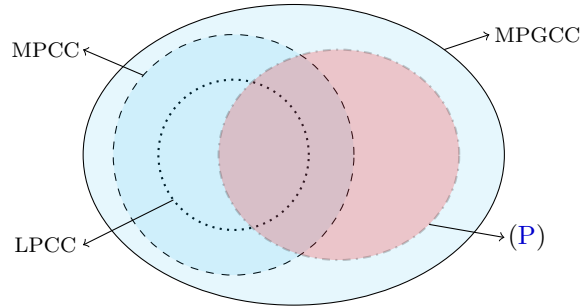
\begin{figure}[htbp]
		\centering
		\begin{tikzpicture}[set/.style={fill=cyan,fill opacity=0.1},scale=2]
			\draw[set,fill=white,dotted,line width=0.8pt] (0,0) ellipse (0.5cm and 0.5cm);
			\draw[set,dashed] (0,0) ellipse(0.8cm and 0.8cm);
			\draw[set] (0.4,0) ellipse(1.4cm and 1cm);
			\draw[set,fill=red,opacity=0.2,line width=0.8pt,dashdotted] (0.7,0) ellipse (0.8cm and 0.7cm);
			\draw[->] (-0.6cm,0.529cm) -- (-1cm,0.7cm);
			\node at (-1.25cm,0.7cm) {\footnotesize MPCC};
			\draw[->] (-0.4cm,-0.3cm) -- (-1cm,-0.6cm);
			\node at (-1.25cm,-0.6cm) {\footnotesize LPCC};
			\draw[->] (1.4cm,0.7cm) -- (1.7cm,0.8cm);
			\node at (2.02cm,0.8cm) {\footnotesize MPGCC};
			\draw[->] (1.3cm,-0.463cm) -- (1.8cm,-0.55cm);
			\node at (1.95cm,-0.55cm) {(\hyperlink{MPGCC}{P})};
		\end{tikzpicture}
		\caption{Relationships among the general LPCC, MPCC, MPGCC as well as the scope of \cref{eqn:MPGCC}.}
		\label{fig:relationship}
	\end{figure}
	
	\par The model \cref{eqn:MPGCC} can be found in various fields. In \cite{hu2021global,hu2022convergence}, the authors formulate the discretized multi-marginal optimal transport problems arising in quantum physics under the so-called Monge-like ansatz into an MPGCC as \cref{eqn:MPGCC}. The generalized complementarity constraints are present to prevent the unfavorable clustering of electrons. As an LPCC or MPCC, \cref{eqn:MPGCC} is able to model sequential decision processes such as pricing in network taxation \cite{brotcorne2000bilevel,brotcorne2001bilevel,labbe1998bilevel,labbe2000class}, biofuel production \cite{bard2000bilevel}, airline industry \cite{cote2003bilevel}, and telecommunication services \cite{bouhtou2009continuous,bouhtou2006pricing,bouhtou2007joint}. 

	\par Similar to \cref{eqn:general MPGCC penalty}, the $\ell_1$ penalized \cref{eqn:MPGCC} reads
	\begin{equation}
		\min_{\z}~f_\beta(\z):=f(\z)+\beta p(\z),\quad\st~\x_i\in\Omega_i^+,~i=1,\ldots,n.
		\tag{P$_\beta$}
		\label{eqn:MPGCC penalty}
	\end{equation}
	Recall that $p$ is defined in \cref{eqn:penalty term}. We denote the feasible region of \cref{eqn:MPGCC penalty} as $\Omega := \prod_{i=1}^n\Omega_i^+$, where ``$\prod$'' stands for the Cartesian product among sets.
		
	\par The aim of this paper is to establish the exactness of the $\ell_1$ penalty function for \cref{eqn:MPGCC} via exploring the relationship between the optimal solution sets of \cref{eqn:MPGCC} and \cref{eqn:MPGCC penalty}. The global optimization algorithms for \cref{eqn:MPGCC} and \cref{eqn:MPGCC penalty} go beyond the range of this work.
	
	\subsection{Literature Review}
	
	\par We review in this part the literatures on the exactness of the $\ell_1$ penalty function for MPGCC. These existing results can be divided into two parts: one is devoted to the general MPCC \cref{eqn:general MPCC}, while the other focuses on the special cases of the model \cref{eqn:MPGCC}. The limitations are gathered at the end of each part, from which we draw our motivation.
	
	\par For the general MPCC \cref{eqn:general MPCC}, one could establish the exactness of the $\ell_1$ penalty function by imposing additional regularity assumptions. In \cite{luo1996mathematical,luo1996exact}, the authors analyze the exactness with the aid of the strict complementarity condition and the following error bound
	\begin{equation}
		\exists~\tau>0,~\st~\dist\lrbracket{(\x,\y),\tilde\calF}\le\tau\inner{\x,\y},\quad\forall~(\x,\y)\in\tilde\Omega,
		\label{eqn:error bound}
	\end{equation}
	where $\tilde\calF$ and $\tilde\Omega$ refer to the feasible regions of \cref{eqn:general MPGCC} and \cref{eqn:general MPGCC penalty}, respectively. Here the strict complementarity condition means that $\x+\y>0$ for any $(\x,\y)\in\tilde\calF$. Later on, the authors of \cite{liu2001exact} establish the exactness under the so-called positive-multiplier nondegeneracy condition and the MFCQ of the penalty problem. The above mentioned regularity assumptions may appear to be restrictive. The strict complementarity condition in \cite{luo1996mathematical,luo1996exact} can easily fail for a general MPCC \cite{hu2021global,liu2001exact}. Moreover, it is nontrivial to check priorly whether the nondegeneracy condition in \cite{liu2001exact} holds at the points of interest. As for the general MPGCC \cref{eqn:general MPGCC}, the theoretical properties of the $\ell_1$ penalty function have not yet been investigated.
	
	\par There are also works dedicated to the $\ell_1$ exact penalty for \cref{eqn:MPGCC} with affine objectives and two variable blocks. The authors of \cite{campelo2000note,campelo2001theoretical} show the exactness based upon the finiteness of extreme point sets; see also earlier works \cite{anandalingam1990solution,labbe1998bilevel,labbe2000class,white1993penalty}, where any optimal solution of \cref{eqn:MPGCC} is proved to solve \cref{eqn:MPGCC penalty} but the reverse direction is ignored. All the works just mentioned concentrate on \cref{eqn:MPGCC} with affine objectives and two variable blocks, while no theoretical results are known for the $n>2$ cases or when the objectives are nonlinear. Nevertheless, the authors of \cite{hu2021global} report rather encouraging numerical results in solving an MPGCC as \cref{eqn:MPGCC} via the $\ell_1$ penalty method. 
   
	\par To sum up, the limitations of the existing works motivate us to pursue the $\ell_1$ penalty exactness result on \cref{eqn:MPGCC} with arbitrary $n\geq 2$ and nonlinear objectives under weaker assumptions than those in \cite{liu2001exact,luo1996mathematical,luo1996exact}. 
	
	\subsection{Contributions}
	
	\par We provide an example of \cref{eqn:MPGCC}, the exactness of whose $\ell_1$ penalty function cannot be implied by the existing results. Leveraging on the special structure of \cref{eqn:MPGCC}, we show the exactness of the $\ell_1$ penalty function under a rather mild assumption. Our results cover those for LPCC in \cite{anandalingam1990solution,campelo2000note,campelo2001theoretical,labbe1998bilevel,labbe2000class,white1993penalty} and applies to the multi-block settings with nonlinear objectives. For a view of our position, please refer to the red ellipsoid in \Cref{fig:relationship}.
	
	\subsection{Notations and Organization}\label{subsec:notation and organization}
	
	\par We denote scalars, vectors, and matrices by lower-case letters, bold lower-case letters, and upper-case letters, respectively. The notations $\one$ and $I$ stand for the all-one vector and identity matrix in proper dimension, respectively. The support of a matrix $X=(X_{ij})$ is presented by $\supp(X):=\{(i,j):X_{ij}\ne0\}$. We use ``$\vv(\cdot)$'' to vectorize matrices by column stacking. The Kronecker product between two matrices is denoted by ``$\otimes$''. The operator ``$\inner{\cdot,\cdot}$'' represents the standard inner product of two vectors or matrices, while ``$\snorm{\cdot}$'' represents the induced norm. 
	
	\par We denote respectively the extreme point sets of $\Omega_i^+$ by $\overline{\Omega_i^+}$ for $i=1,\ldots,n$, the extreme point set of $\Omega$ by $\bar\Omega$, the optimal solution set of \cref{eqn:MPGCC} by $S^{\opt}$, and the optimal solution set of \cref{eqn:MPGCC penalty} by $S_\beta^{\opt}$ for any $\beta\in\R$. Let $\bar S^{\opt}:=\bar\Omega\cap S^{\opt}$ and $\bar S_\beta^{\opt}:=\bar\Omega\cap S_\beta^{\opt}$ denote the extreme-point optimal solution of \cref{eqn:MPGCC} and \cref{eqn:MPGCC penalty}, respectively. The cardinality of a set $S$ is represented by $\abs{S}$. The line segment connected by two points, $\mb{a}$ and $\mb{b}$, is denoted by $\ell(\mb{a},\mb{b}):=\{(1-\lambda)\mb{a}+\lambda\mb{b}:\lambda\in[0,1]\}$. The notation $\rbd S$ refers to the relative boundary of a set $S$. The distance between a point $\mb{a}$ and a closed set $S$ is represented by $\dist(\mb{a},S):=\min_{\w\in S}\snorm{\mb{a}-\w}$. 
	
	\par We organize this paper as follows. The example of \cref{eqn:MPGCC} falling outside the literatures is detailed in \cref{sec:instance}. The main results of the $\ell_1$ penalty exactness on \cref{eqn:MPGCC} are elaborated in \cref{sec:main results}. We draw conclusions and perspectives in \cref{sec:conclusions}.
	
	\section{The $\ell_1$ Penalty Exactness beyond the Existing Works}\label{sec:instance}
	
	\par We give an instance of \cref{eqn:MPGCC}, related to real applications \cite{hu2021global,hu2022convergence}, whose $\ell_1$ penalty exactness cannot be implied by the existing works. Specifically, it is not an LPCC analyzed in \cite{anandalingam1990solution,campelo2000note,campelo2001theoretical,labbe1998bilevel,labbe2000class,white1993penalty}. Moreover, neither the strict complementarity condition nor the error bound \cref{eqn:error bound} in \cite{luo1996mathematical,luo1996exact} is satisfied. 
	
	\par Consider the following MPCC with three matrix variables
	\begin{equation}
		\begin{array}{cl}
			\min\limits_{Z} & \displaystyle\sum_{i=1}^3\inner{X_i,C}+\sum_{i=1}^3\sum_{j>i}\inner{X_i,X_jC}\\
			\st & X_i\one=\one,~X_i^\T\one=\one,~\trace(X_i)=0,~X_i\ge0,~i=1,2,3,\\
			& \inner{X_1,X_2}=0,~\inner{X_1,X_3}=0,~\inner{X_2,X_3}=0,
		\end{array}
		\label{eqn:example}
	\end{equation}
	where $K\in\N$, $X_1$, $X_2$, $X_3\in\R^{K\times K}$, $Z:=(X_1,X_2,X_3)\in\prod_{i=1}^3\R^{K\times K}$, $C=(C_{ij})\in\R^{K\times K}$ is defined as
	$$C_{ij}=\left\{\begin{array}{ll}
		\dfrac{1}{\abs{i-j}h}, & \text{if}~i\ne j,\\
		0,                   & \text{otherwise}
	\end{array}\right.~\text{with}~h:=\frac{1}{K}.$$
	The problem \cref{eqn:example} is in the form of \cref{eqn:MPGCC} with $n=3$, $m=K^2$, $\x_i=\vv(X_i)$ ($i=1,2,3$), $f(\z)=\sum_{i=1}^3\inner{\mb{a},\x_i}+\sum_{i=1}^3\sum_{j>i}\inner{\x_i,R\x_j}$, $\mb{a}=\vv(C)$, $R=\frac{1}{2}\begin{bmatrix}
		0 & C\otimes I\\
		C\otimes I & 0
	\end{bmatrix}$, and 
	$$\Omega_1=\Omega_2=\Omega_3=\lrbrace{\w\in\R^{K^2}:\begin{bmatrix}
		\one^\T\otimes I\\
		I\otimes\one^\T \\
		\vv(I)^\T
	\end{bmatrix}\w=\begin{bmatrix}
		\one\\\one\\0
	\end{bmatrix}},~\calF=\lrbrace{\z\in\prod_{i=1}^3\Omega_i^+:p(\z)=0}.$$
	This instance is derived from the one-dimensional multi-marginal optimal transport model arising in many-body quantum physics \cite{hu2021global,hu2022convergence}, particularly with four electrons, constant density, and uniform discretization over the interval $[0,1]$. By \cite{colombo2015multimarginal}, if $K$ can be divided by $4$, we are able to explicitly write down one optimal solution $Z^*:=(X_1^*,X_2^*,X_3^*)$ to \cref{eqn:example}:
	\begin{align}
		(X_1^*)_{ij}&=\left\{\begin{array}{ll}
			1, & \text{if}~i-j=\frac{3K}{4}~\text{or}~j-i=\frac{K}{4},\\
			0, & \text{otherwise},
		\end{array}\right.\nonumber\\
		(X_2^*)_{ij}&=\left\{\begin{array}{ll}
			1, & \text{if}~i-j=\frac{K}{2},~\text{or}~j-i=\frac{K}{2},\\
			0, & \text{otherwise},
		\end{array}\right.\label{eqn:optimal solution}\\
		(X_3^*)_{ij}&=\left\{\begin{array}{ll}
			1, & \text{if}~i-j=\frac{K}{4}~\text{or}~j-i=\frac{3K}{4},\\
			0, & \text{otherwise}.
		\end{array}\right.\nonumber
	\end{align}
	An illustration of $Z^*$ can be found in \Cref{fig:solution}.
	
	\begin{figure}[htbp]
		\centering
		\includegraphics[width=.9\textwidth]{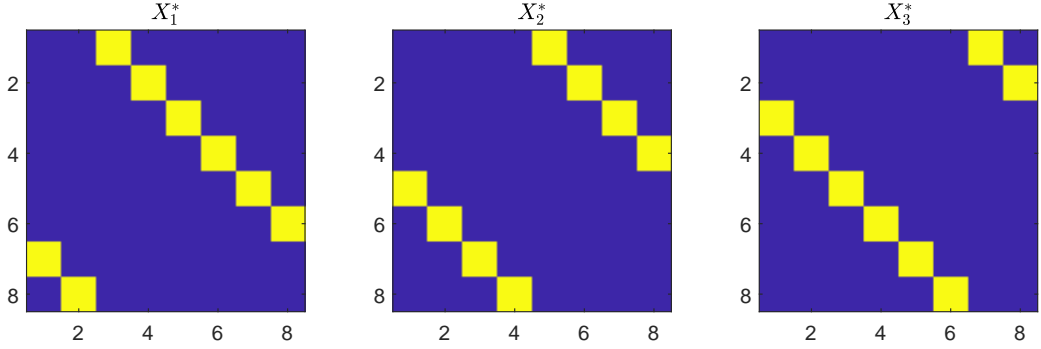}
		\caption{The illustration of $Z^*$ when $K=8$. \textit{Each yellow block stands for the value of $1$, while each deep blue block for the value of $0$.}}
		\label{fig:solution}
	\end{figure}
	
	\par Since $n>2$ and $f$ is nonlinear, \cref{eqn:example} cannot be covered by the existing works on the special cases of \cref{eqn:MPGCC}. Moreover, the strict complementarity condition is violated on \cref{eqn:example} due to the zero-trace constraints. Next, we argue that the error bound \cref{eqn:error bound} fails to hold on \cref{eqn:example} whenever $K>4$ and $\bmod(K,4)=0$. 
	
	\begin{thm}\label{thm:error bound fail}
		For \cref{eqn:example} with $K:K>4$ and $\bmod(K,4)=0$, there does not exist $\tau>0$ such that\footnote{With a slight abuse of notation, we still use $p$ as the $\ell_1$ penalty term with matrix variables.}
		\begin{equation}
			\dist(Z,\calF)\le\tau p(Z),\quad\forall~Z\in\Omega.
			\label{eqn:error bound new}
		\end{equation}
	\end{thm}
	
	\begin{proof}
		In light of the optimal solution $Z^*$ in \cref{eqn:optimal solution}, for any $\eps\in(0,1)$, we come up with $Z(\eps):=(X_1(\eps),X_2^*,X_3(\eps))$ (see \Cref{fig:counterexm} for illustration) where $X_1(\eps)$ and $X_3(\eps)$ are defined as
		\begin{align*}
			(X_1(\eps))_{ij}&=\left\{\begin{array}{ll}
				1-\eps, & \text{if}~i-j=\frac{3K}{4}~\text{or}~j-i=\frac{K}{4},\\
				\eps,   & \text{if}~i-j=\frac{3K}{4}-1~\text{or}~j-i=\frac{K}{4}+1,\\
				0,      & \text{otherwise},
			\end{array}\right.\\
			(X_3(\eps))_{ij}&=\left\{\begin{array}{ll}
				1-\eps, & \text{if}~i-j=\frac{K}{4}~\text{or}~j-i=\frac{3K}{4},\\
				\eps,   & \text{if}~i-j=\frac{3K}{4}-1~\text{or}~j-i=\frac{K}{4}+1,\\
				0,      & \text{otherwise}.
			\end{array}\right.
		\end{align*}
		It is easy to verify that $Z(\eps)\in\Omega$ but $p(Z(\eps))=K\eps^2$ for any $\eps\in(0,1)$ (since $K>4$), and $Z(\eps)\to Z^*$ as $\eps\to0+$. 
		
		\begin{figure}[htbp]
			\centering
			\includegraphics[width=.6\textwidth]{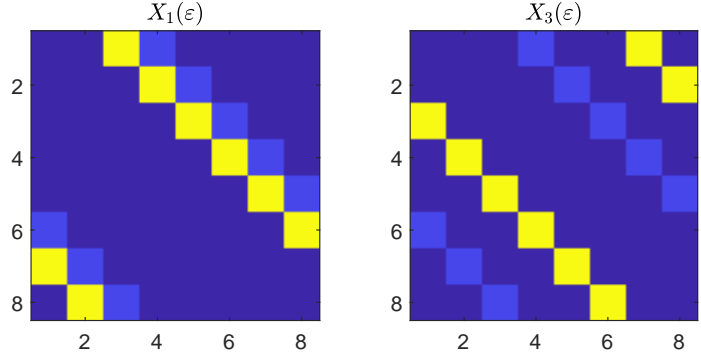}
			\caption{The illustration of $X_1(\eps)$ and $X_3(\eps)$ when $K=8$. \textit{Each yellow block stands for the value of $1-\eps$, each light blue block for the value of $\eps$, while each deep blue block for the value of $0$.}}
			\label{fig:counterexm}
		\end{figure}
		
		\par If it happens that $Z^*\in\argmin_{Z\in\calF}\snorm{Z-Z(\eps)}$, then 
		$$\dist(Z(\eps),\calF)=\snorm{Z^*-Z(\eps)}=2\sqrt{K}\eps=\frac{2}{\sqrt{K}\eps}p(Z(\eps)).$$
		Passing $\eps\to0+$ yields the nonexistence of a uniform $\tau>0$ in \cref{eqn:error bound new}.
		
		\par If $Z^*\notin\argmin_{Z\in\calF}\snorm{Z-Z(\eps)}$, we will show that 
		\begin{equation}
			\limsup_{\eps\to0+}\frac{\dist(Z(\eps),\calF)}{\eps}>0,
			\label{eqn:tmp1}
		\end{equation}
		leading again to the failure of \cref{eqn:error bound new}. Suppose otherwise that 
		\begin{equation}
			\dist(Z(\eps),\calF)=\snorm{\tilde Z(\eps)-Z(\eps)}=o(\eps)~\text{as}~\eps\to0+,
			\label{eqn:tmp2}
		\end{equation}
		where $\tilde Z(\eps):=(\tilde X_1(\eps),\tilde X_2(\eps),\tilde X_3(\eps))\in\calF$. Since the nonzero entries of $X_1(\eps)$ and $X_3(\eps)$ are either $1-\eps$ or $\eps$ and, by \cref{eqn:tmp2}, $\abs{(X_i(\eps))_{jk}-(\tilde X_i(\eps))_{jk}}=o(\eps)$ for $i=1,3$ and $j,k=1,\ldots,K$, we know that $\supp(\tilde X_i(\eps))\supseteq\supp(X_i(\eps))$ for $i=1,3$. Note that $\tilde X_1(\eps)$, $\tilde X_3(\eps)\ge0$. Hence, $\inner{\tilde X_1(\eps),\tilde X_3(\eps)}>0$ because $\inner{X_1(\eps),X_3(\eps)}>0$. This contradicts with $\tilde Z(\eps)\in\calF$ and validates \cref{eqn:tmp1}. Combining the above two cases, we complete the proof.
	\end{proof}

	As a result, the exactness of the $\ell_1$ penalty function for \cref{eqn:MPGCC} cannot be achieved as in the literatures \cite{anandalingam1990solution,campelo2000note,campelo2001theoretical,labbe1998bilevel,labbe2000class,luo1996mathematical,luo1996exact,white1993penalty}. We next rigorously prove this exactness under a rather mild assumption by exploiting the structure of \cref{eqn:MPGCC}.
	
	\section{Main Results}\label{sec:main results}
	
	\par In this part, we elaborate the proof of $\ell_1$ penalty exactness on \cref{eqn:MPGCC}. For a reference of notations, please look back to \cref{subsec:notation and organization}. We make the following weak assumption on $\Omega$ hereafter. 
	
	\begin{assume}\label{assume:compactness}
		The set $\Omega$ is nonempty and compact. 
	\end{assume}

	\par \Cref{assume:compactness} usually automatically holds in real applications \cite{hu2021global,labbe1998bilevel}. With this assumption, the attainment of the optimal value of \cref{eqn:MPGCC penalty} for any $\beta\in\R$ is ensured. 

	\par The first result states that, for any $\beta\in\R$, $\bar S^{\opt}_{\beta}\ne\emptyset$, i.e., there exists at least one extreme-point optimal solution for \cref{eqn:MPGCC penalty}. Its proof hinges on the multi-affine structure of $f_\beta$ and the separability of $\Omega$ with respect to variable blocks. 
	
	\begin{lem}\label{lem:extreme optimal}
		Suppose \Cref{assume:compactness} holds. Then $\bar S^{\opt}_{\beta}\ne\emptyset$ for any $\beta\in\R$. 
	\end{lem}
	\begin{proof}
		By \Cref{assume:compactness}, $S_{\beta}^{\opt}\ne\emptyset$ for any $\beta\in\R$. Pick $\hat\z:=(\hat\x_1,\ldots,\hat\x_n)\in S^{\opt}_{\beta}$. If it happens that $\hat\z\in \bar\Omega$, we are done. Otherwise, consider a linear program
		\begin{equation}
			\displaystyle\min_{\x_1\in\Omega_1^+}~~f_\beta(\x_1,\hat\x_2,\ldots,\hat\x_n).
			\label{eqn:linear program}
		\end{equation}
		Let $\x_1^*\in\overline{\Omega_1^+}$ be its extreme-point optimal solution, whose existence is guaranteed by Bauer maximum principle \cite{bauer1958minimalstellen} and \Cref{assume:compactness}. Then by the optimality of $\x_1^*$ and the feasibility of $\hat\x_1$, one has $f_\beta(\x_1^*,\hat\x_2,\ldots,\hat\x_n)\le f_\beta(\hat\x_1,\ldots,\hat\x_n)$. Next, for any $j\ge2$, consider successively linear program of the form
		$$\displaystyle\min_{\x_j\in\Omega_j^+}~~f_\beta(\x_1^*,\ldots,\x_{j-1}^*,\x_j,\hat\x_{j+1},\ldots,\hat\x_n),$$
		and let $\x_j^*\in\overline{\Omega_j^+}$ be its extreme-point optimal solution, whose existence is ensured similarly. Clearly, it holds, for any $j\in\{2,\ldots,n\}$, that 
		$$f_\beta(\x_1^*,\ldots,\x_{j-1}^*,\x_j^*,\hat\x_{j+1},\ldots,\hat\x_n)\le f_\beta(\x_1^*,\ldots,\x_{j-1}^*,\hat\x_j,\hat\x_{j+1},\ldots,\hat\x_n).$$
		Together, we have a chain
		$$f_\beta(\x_1^*,\ldots,\x_n^*)\le f_\beta(\x_1^*,\ldots,\hat\x_n)\le\cdots\le f_\beta(\x_1^*,\hat\x_2,\ldots,\hat\x_n)\le f_\beta(\hat\x_1,\ldots,\hat\x_n).$$
		Note that due to the separability of $\Omega$, $\bar\Omega=\prod_{i=1}^n\overline{\Omega_i^+}$ and $\z^*:=(\x_1^*,\ldots,\x_n^*)\in \bar\Omega$. Moreover, since $\hat\z$ is an optimal solution of \cref{eqn:MPGCC penalty}, one has $\z^*\in S^{\opt}_{\beta}$ in view of the last inequality chain. This completes the proof. 
	\end{proof}
	
	\par Utilizing the finiteness of $\bar\Omega$ and \Cref{lem:extreme optimal}, we could show the following partial exact penalty result. 
	
	\begin{thm}\label{thm:partial exact penalty}
		Suppose \Cref{assume:compactness} holds. Then there exists $\bar\beta>0$ such that $\bar S^{\opt}_{\beta}\subseteq S^{\opt}\subseteq S^{\opt}_{\beta}$ for any $\beta\ge\bar\beta$.
	\end{thm}

	\begin{proof}
		We begin with the proof of $\bar S_\beta^{\opt}\subseteq S^{\opt}$ for all sufficiently large $\beta$ by contradiction. Suppose that there exists $\{\beta^{(k)}>0\}\to\infty$ and a corresponding $\{\z^{(k)}\}\subseteq \bar S_{\beta^{(k)}}^{\opt}\setminus S^{\opt}$ (due to \Cref{lem:extreme optimal}). Clearly, if $\z\in S_\beta^{\opt}$ for some $\beta>0$ and $p(\z)=0$, then $\z\in S^{\opt}$. It then holds that $p(\z^{(k)})>0$ for each $k$. Since $\{\Omega_i^+\}_{i=1}^n$ are polyhedral, one has $\abs{\bar\Omega}<\infty$ and
		\begin{equation}
			\inf_{l}f(\z^{(l)})>-\infty~\text{and}~\inf_lp(\z^{(l)})>0.
			\label{eqn:finiteness benefit}
		\end{equation}
		Given any $\z$ feasible for \cref{eqn:MPGCC}, it holds, for any $k$, that
		\begin{align*}
			\infty>f(\z)&=f(\z)+\beta^{(k)}p(\z)\ge f_{\beta^{(k)}}(\z^{(k)})\nonumber\\
			&=f(\z^{(k)})+\beta^{(k)}p(\z^{(k)})\\
			&\ge\inf_lf(\z^{(l)})+\beta^{(k)}\inf_lp(\z^{(l)}).\nonumber
		\end{align*}
		Based upon \cref{eqn:finiteness benefit}, we reach a contradiction after passing $k$ to $\infty$ in the last inequality. Therefore, there exists $\bar\beta>0$ such that $\bar S_\beta^{\opt}\subseteq S^{\opt}$ holds for any $\beta\ge\bar\beta$.
		
		\par Now we show the second inclusion, i.e., given $\beta\ge\bar\beta$, for any $\z^*\in S^{\opt}$ and $\z\in\Omega$, $f_\beta(\z^*)\le f_\beta(\z)$. By \Cref{lem:extreme optimal}, $\bar S_\beta^{\opt}\ne\emptyset$. Pick $\hat\z\in \bar S_\beta^{\opt}$. From the first inclusion, we know that $\hat\z\in S^{\opt}$. Hence $f_\beta(\z^*)=f(\z^*)=f(\hat\z)=f_\beta(\hat\z)$. Since $\hat\z\in S^{\opt}_\beta$, we further have $f_\beta(\hat\z)\le f_\beta(\z)$. Combining with the last equality chain, one concludes that $f_\beta(\z^*)\le f_\beta(\z)$. 
	\end{proof}

	To complete the exactness, it suffices to show the reverse direction: $S_\beta^{\opt}\subseteq S^{\opt}$. We begin with the following lemma, which states that the optimal solution of \cref{eqn:MPGCC penalty} with only one non-extreme variable block must be optimal for \cref{eqn:MPGCC}.

	\begin{lem}\label{lem:one not extreme}
		Suppose \Cref{assume:compactness} holds. For a given $\beta\ge\bar\beta$, let $\hat\z:=(\hat\x_1,\ldots,\hat\x_n)\in S_{\beta}^{\opt}$. If there exists only one index $i\in\{1,\ldots,n\}$ such that $\hat\x_i\notin\overline{\Omega_i^+}$, then $\hat\z\in S^{\opt}$. 
	\end{lem}

	\begin{proof}
		Without loss of generality, we assume $i=1$. Suppose, on the contrary, that $\hat\z\notin S^{\opt}$. Then we must have $p(\hat\z)>0$. Since $\hat\z\in S_{\beta}^{\opt}$, $\hat\x_1$ solves the linear program \cref{eqn:linear program} and $\hat\x_1\in\rbd\Omega_1^+$. Due to $\hat\x_1\notin\overline{\Omega_1^+}$, \cref{eqn:linear program} is degenerate on a face, say $\hat F_1$, of $\Omega_1^+$ where $\hat\x_1$ lies and, therefore, each point in $\hat F_1$ solves \cref{eqn:linear program}. In particular, by \Cref{thm:partial exact penalty}, each extreme point of $\hat F_1$, together with $\{\hat\x_j\}_{j=2}^n$, solves \cref{eqn:MPGCC}. Denote by $\{\bar\x_1^{(i)}\}_{i\in\hat I_1}$ the extreme points of $\hat F_1$, where $\hat I_1\subseteq\N$ is a finite index set. Note that $\hat F_1$ is convex compact. By Minkowski's theorem (see, e.g., \cite[Corollary 18.5.1]{Rockafellar+1970}), $\hat F_1$ is the convex hull of $\{\bar\x_1^{(i)}\}_{i\in\hat I_1}$. Since $\hat\x_1\notin\overline{\Omega_1^+}$, $\abs{\hat I_1}\ge2$.
		
		\par Pick any two elements in $\hat I_1$, e.g., $i=1,2$. Then $(\bar\x_1^{(i)},\hat\x_2,\ldots,\hat\x_n)\in S^{\opt}$ and $p(\bar\x_1^{(i)},\hat\x_2,\ldots,\hat\x_n)=0$ for $i=1,2$. Since the function $p$ is linear with respect to the first block, one must have $p(\x_1,\hat\x_2,\ldots,\hat\x_n)=0$ for any $\x_1\in\ell(\bar\x_1^{(1)},\bar\x_1^{(2)})$.
		
		\par Invoking the arguments in the former paragraph repeatedly, we could deduce that the function $p\equiv0$ on $\rbd\hat F_1$. If $\hat\x_1\in\rbd\hat F_1$, we then reach a contradiction. Otherwise, since $\hat F_1$ is convex compact, $\hat\x_1$ must lies on the line segment connected by two points in $\rbd\hat F_1$. Again, since $p$ is linear with respect to the first block, $p(\hat\z)=0$, leading to a contradiction. The proof is complete.
	\end{proof}

	We then prove the exactness result by mathematical induction. 
	
	\begin{thm}\label{thm:exact penalty}
		Suppose \Cref{assume:compactness} holds.  For a given $\beta\ge\bar\beta$, $S_{\beta}^{\opt}=S^{\opt}$.
	\end{thm}

	\begin{proof}
		\par The ``$\supseteq$'' part is ensured by \Cref{thm:partial exact penalty}. Below, we deal with the ``$\subseteq$'' part. We prove by mathematical induction that, for $r\in\{0,1,\ldots,n\}$, any $\hat\z\in S_{\beta}^{\opt}$ with only $r$ indice $\{i_1,\ldots,i_r\}\subseteq\{1,\ldots,n\}$ for which $\hat\x_{i_s}\notin\overline{\Omega_{i_s}^+}$, $s=1,\ldots,r$, must belong to $S^{\opt}$. It is easy to check that if this claim holds true, the desired result will follow. The cases $r=0$ and $r=1$ is valid in view of \Cref{thm:partial exact penalty} and \Cref{lem:one not extreme}, respectively. Now suppose the claim is valid for $r=t$ where $t\in\{1,\ldots,n-1\}$ and we show the $r=t+1$ case. 
		
		\par Without loss of generality, we assume $i_s=s$ for $s=1,\ldots,t+1$. Suppose otherwise $\hat\z\notin S^{\opt}$. Then we have $p(\hat\z)>0$. Since $\hat\z\in S_{\beta}^{\opt}$, $\hat\x_{t+1}$ solves the following linear program
		\begin{equation}
			\min_{\x_{t+1}\in\Omega_{t+1}^+}~~f_\beta(\hat\x_1,\ldots,\hat\x_t,\x_{t+1},\hat\x_{t+2},\ldots,\hat\x_n).
			\label{eqn:linear program xt+1}
		\end{equation}
		Due to $\hat\x_{t+1}\notin\overline{\Omega_{t+1}^+}$, \cref{eqn:linear program xt+1} is degenerate on a face, say $\hat F_{t+1}$, of $\Omega_{t+1}^+$ where $\hat\x_{t+1}$ lies and, therefore, each point in $\hat F_{t+1}$ solves \cref{eqn:linear program xt+1}. By the optimality of objective value, each extreme point of $\hat F_{t+1}$, together with $\{\hat\x_j\}_{j\ne t+1}$, solves \cref{eqn:MPGCC penalty}. Denote by $\{\bar\x_{t+1}^{(i)}\}_{i\in\hat I_{t+1}}$ the extreme points of $\hat F_{t+1}$, where $\hat I_{t+1}\subseteq\N$ is a finite index set, and let $\bar\z^{(i)}:=(\hat\x_1,\ldots,\hat\x_t,\bar\x_{t+1}^{(i)},\hat\x_{t+2},\ldots,\hat\x_n)$ for any $i\in\hat I_{t+1}$. Then $\{\bar\z^{(i)}\}_{i\in\hat I_{t+1}}\subseteq S_{\beta}^{\opt}$. For each $i\in\hat I_{t+1}$, notice that there are only $t$ variable blocks in $\bar\z^{(i)}$ that do not lie in their corresponding extreme point sets. By induction, we get $\{\bar\z^{(i)}\}_{i\in\hat I_{t+1}}\subseteq S^{\opt}$ and, therefore, $p(\bar\z^{(i)})=0$ for any $i\in\hat I_{t+1}$. By similar arguments on $\hat F_{t+1}$ as in \Cref{lem:one not extreme}, we could arrive a contradiction. Hence, the claim is true for $r=t+1$. 
		
		\par By mathematical induction, the claim holds true for any $r\in\{0,1,\ldots,n\}$, which completes the proof.
	\end{proof}

	\section{Conclusions and Perspectives}\label{sec:conclusions}
	
	\par In this paper, we prove the exactness of the $\ell_1$ penalty function for the model \cref{eqn:MPGCC}. Without restrictive assumptions such as the strict complementarity condition or the positive-multiplier nondegeneracy condition in the literatures, we manage to show $S^{\opt}= S^{\opt}_\beta$ for all sufficiently large $\beta$ by exploiting the multi-affine structure of the objective function and generalized complementarity constraints. Our results cover those existing ones for LPCC and apply to the multi-block settings with nonlinear objective functions. We also put forward an instance of \cref{eqn:MPGCC}, the exactness of whose $\ell_1$ penalty function cannot be implied by the existing results, but is ensured by ours.
	
	\par The computational aspects of the model \cref{eqn:MPGCC} deserve future investigation. Specifically, in virtue of the applications in \cite{hu2021global}, it is essential to design tailor-made efficient solvers that produce approximate optimal solutions for \cref{eqn:MPGCC penalty}.
	
	\par Moreover, the results stated in this paper can be easily extended to a class of multi-affine optimization problems in the form
	$$\begin{array}{cl}
		\min\limits_{\z} & \tilde f(\x_1,\ldots,\x_n)\\
		\st & \x_i\in\Omega_i,~i=1,\ldots,n,\\
		 & h_j(\x_1,\ldots,\x_n)=0,~j=1,\ldots,q,
	\end{array}$$
	where $\tilde f$, $h_j:\R^{mn}\to\R$ ($j=1,\ldots,p$) are multi-affine, $\{\Omega_i\}_{i=1}^n$ are polyhedrons; for $j=1,\ldots,q$, $h_j\ge0$ over $\prod_{i=1}^n\Omega_i$. The related $\ell_1$ penalty function $\tilde f_\beta(\z)=\tilde f(\z)+\beta\sum_{j=1}^qh_j(\z)$. For more details on the general multi-affine optimization, see \cite{drenick1992multilinear,gao2020admm} and the references therein.
	
	\section*{Acknowledgements}
	
	The authors would like to thank Professor Huajie Chen for his discussions on the applications in quantum physics. This work was supported by the National Natural Science Foundation of China (Grant Nos. 12125108, 11971466, 11991021, 11991020, 12021001, and 12288201), Key Research Program of Frontier Sciences, Chinese Academy of Sciences (Grant No. ZDBS-LY-7022), and the CAS AMSS-PolyU Joint Laboratory in Applied Mathematics.
	
	\section*{Declaration of Competing Interest}
	
	\par The authors declare that they have no conflicts of interest in this work.
	
	\normalem
	\bibliographystyle{plain}
	\bibliography{ref}
\end{document}